\documentclass[a4paper,12pt]{article}

\usepackage{amsmath}
\usepackage{amsthm}
\usepackage{dsfont}
\usepackage[T1]{fontenc}
\usepackage{cite}
\usepackage{latexsym,amssymb,amscd}
\usepackage{amsfonts}
\usepackage[mathscr]{eucal}
\usepackage{enumerate}
\usepackage{longtable}
\usepackage{tikz}
\usepackage{multirow}
\usepackage{authblk}

\usetikzlibrary{arrows,snakes,backgrounds}
\usetikzlibrary{matrix}
\tikzstyle{point}=[circle,draw,inner sep=0pt,minimum size=4mm]
\tikzstyle{ordinario}=[circle,draw,inner sep=0pt,minimum size=3mm]

\theoremstyle{plain}
 \newtheorem{theorem}{Theorem}

 \newtheorem{corollary}[theorem]{Corollary}
\theoremstyle{remark}
 
\theoremstyle{definition}

\begin{document}
\bibliographystyle{apalike}

\title{Deducing factoring methods through concrete material}

\author{Ivon Dorado and Ricardo Torres}
\date{}
\maketitle

\begin{abstract}
\noindent {\footnotesize We formulate and prove a criterion for reducibility of a quadratic polynomial over the integers. The main theorem was suggested by the teaching experience with the concrete material called "the polynomial box" \cite{SMGpolyBox}. Through the corollaries we relate our theorem and the use of concrete material with some well know factoring methods for quadratic polynomial with integer coeficients.
}
\end{abstract}

 {\small {\it 2010 Mathematics Subject Classification}: 97B50, 97A80, 97H40, 12D05.\\

 {\small {\it Keywords}: Polynomial box, concrete material, factorization criteria, quadratic polynomials.

\section{Introduction}\label{1}

Since the beginning of algebra, the development of algebraic operations was based on solving real-life problems using geometry and different concrete representations that allowed to simplify some situations. The use of those elements has played a fundamental role in the construction of mathematical knowledge.

To take into account those remarkable facts of the historical and cultural legacy of mathematics can be quite useful to contribute to the teaching of mathematics and its study.

The present paper arises from the research entitled 
\textit{Teaching of the operations with first and second degree polynomials in one variable under the approach of solving problems} in which the inductive and metacognitive approach exposed by Mason, Burton and Stacey \cite{MBS1982} was applied as a teaching strategy and also the use and comparison of different languages such as: concrete, supported by the \textit{polynomial box} \cite{SMGpolyBox}, pictorial and algebraic.

During the development of teaching material: workshop guides, and different problems, it was possible to infer a relation between the traditional factorization methods for reducible quadratic polynomials and the concrete representation that was worked with the polynomial box. The results of these observations are presented below.

\section{Main results}

This paper is the product of a research which aimed to teach algebraic operations by using a concrete material known as the \textit{polynomial box} to students of a public highschool in Bogot\'a, Colombia.

Factoring a quadratic polynomial with the polynomial box is presented to the students like a puzzle, where they have to build a rectangle with \textit{cards} like in the figure

$$
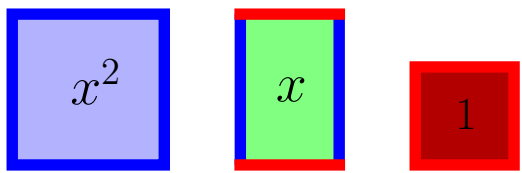
\begin{tikzpicture}
\fill[blue!30!white] (0,0) rectangle (2,2);
\draw[line width=1.5mm,blue] (0,0) rectangle (2,2);
\node (x2) at (1.1,1.1) {\LARGE $x^2$};
\fill[green!50!white] (3,0) rectangle (4.3,2);
\draw[line width=1.5mm,blue] (3,0) rectangle (3,2);
\draw[line width=1.5mm,blue] (4.3,0) rectangle (4.3,2);
\draw[line width=1.5mm,red] (4.3775,2) -- (2.922,2);
\draw[line width=1.5mm,red] (4.3775,0) -- (2.922,0);
\node (x1) at (3.67,1) {\LARGE $x$};
\fill[red!70!black] (5.3,0) rectangle (6.6,1.3);
\draw[line width=1.5mm,red] (5.3,0) rectangle (6.6,1.3);
\node (one) at (5.97,0.67) {\Large $1$};
\end{tikzpicture}
$$

For doing this, you just have to follow a simple rule: when you put a card next to another, the sides that will be adjacent must have the same length.

The cards in the figure were adapted from the cards given by Soto, Mosquera and Gomez \cite{SMGpolyBox}. The sides of each rectangle were color in order to highlight the difference between the length of the side and the area of the card. This was also useful to illustrate the students the rule of the adjacent sides (for more information about the use of polynomial box, see \cite{SMGpolyBox}).

Out of this research it was possible to rediscover different ways to factoring quadratic polynomials which turn into the following theorems and corollaries.

\begin{theorem}\label{main}
Let $f(x)=ax^2 + bx + c$ be a polynomial with integer coefficients and $a\neq 0$. Then $f(x)$ can be factorized over $\mathbb{Z}$ in two linear polynomials if and only if there exist $p, q \in \mathbb{Z}$ such that $pq=ac$ and $p+q=b$. 

In this case, the roots of $f(x)$ are $-\dfrac{p}{a}$ and $-\dfrac{q}{a}$.
\end{theorem}

\begin{proof}
Consider 
$$f(x)=(p_1x+p_2)(q_1x+q_2);$$
then 
$$a=p_1q_1, \hspace{5mm} b=p_1q_2+p_2q_1 \hspace{5mm}\text{ and } \hspace{5mm}c=p_2q_2.$$

We define $p=p_1q_2$ and $q=p_2q_1$, in such a way that the required conditions are satisfied.

Conversely, suppose that $pq=ac$ and $p+q=b$, for some $p, q \in \mathbb{Z}$. We denote
$$p_1= m.c.d. (p,a).$$

There exist two integer $q_1$ and $q_2$ such that $a=p_1q_1$ and $p=p_1q_2$. We have
$$p_1q_2q=p_1q_1c;$$
as $a\neq 0$, then $p_1\neq 0$ and $q_2q=q_1c$.

If $p=0$, 
$$p_1=a, \hspace{5mm} q_1=1 \hspace{5mm}\text{ and } \hspace{5mm} q_2=0.$$

In any case, notice that $m.c.d. (q_1,q_2) = 1$. Therefore $q_1 | q$, so there exists $p_2\in \mathbb{Z}$ such that 
$q= q_1p_2.$ Then
$$q_2q_1p_2=q_1c;$$
and $p_2q_2=c$.

We have
$$f(x)=ax^2 + bx + c$$
$$= p_1q_1 x^2 + (p_1q_2 + q_1p_2)x + p_2q_2$$
$$= (p_1x+p_2)(q_1x+q_2);$$
which proves the equivalence.

Now consider the existence of $p$ and $q$ as required, and the quadratic formula to find the roots of $f(x)$. The discriminant of the equation is a square integer, namely
$$b^2-4ac = (p+q)^2 - 4pq$$
$$=p^2+2pq+q^2-4pq$$
$$=(p-q)^2.$$

Then the roots of the polynomial are given by
$$\frac{-b\pm |p-q|}{2a};$$
but $b=p+q$, so the roots of $f(x)$ are
$$-\dfrac{p}{a}\ \ \ \text{ and }\ \ \  -\dfrac{q}{a};$$
finishing our proof.
\end{proof}

This theorem is specially useful to formulate factoring exercises, because given two no null integers $p$ and $q$, by Theorem \ref{main}, the polynomial $px^2+(p+q)x+q$ and its reciprocal $qx^2+(p+q)x+p$ are factorable.

In fact, the reciprocal polynomial criterion and the Eisenstein criterion for irreducibility of a polynomial of the form $ax^2+bx+c \in \mathbb{Z}[x]$ can be viewed as corollaries of our main theorem.

\begin{corollary}{Reciprocal polynomial criterion.}
A quadratic polynomial 
$$ax^2+bx+c \in \mathbb{Z}[x];$$ 
can be factorized over $\mathbb{Z}$ in two linear polynomials if and only if its reciprocal polynomial 
$$cx^2+bx+a \in \mathbb{Z}[x];$$ 
is factorable over $\mathbb{Z}$ in two linear factors.
\end{corollary}

\begin{proof}
The polynomial $ax^2+bx+c$ is factorable if and only if there exist $p, q \in \mathbb{Z}$ such that $pq=ac$ and $p+q=b$, by Theorem \ref{main}, this is equivalent to the polynomial $cx^2+bx+a$ is factorable.
\end{proof}

\begin{corollary}{Eisenstein criterion.}
Let $f(x)=ax^2+bx+c$ a polynomial with integer coefficients and $a\neq 0$. If there exists a prime number $r$ which is a divisor of $b$ and $c$, but does not divide $a$, and such that $r^2$ does not divide $c$, then $f(x)$ is irreducible over $\mathbb{Q}$. 
\end{corollary}

\begin{proof}
If there are two integers $p$ and $q$ such that $pq=ac$, as $r|c$ but $r\nmid a$ and $r^2 \nmid c$, then
$r|p$ or $r|q$ but not both. Therefore $r\nmid p+q$.

By the hypothesis, $r$ is a divisor of $b$, so there are no integers $p$ and $q$ such that $pq=ac$ and $p+q=b$. Using Theorem \ref{main}, the polynomial $f(x)$ can not be factorized in two linear factors over $\mathbb{Z}$, then $f(x)$ does not have rational roots, so it is irreducible over $\mathbb{Q}$ because it is a quadratic polynomial. 
\end{proof}

Theorem \ref{main} was suggested by the use of the polynomial box because to factorize the polynomial $ax^2+bx+c$ with it, we need to construct two rectangles with only one common vertex: one with $|a|$ cards of type $x^2$, and the other with $|c|$ cards of type 1. We obtain a figure like one of the following or its rotations 

$$
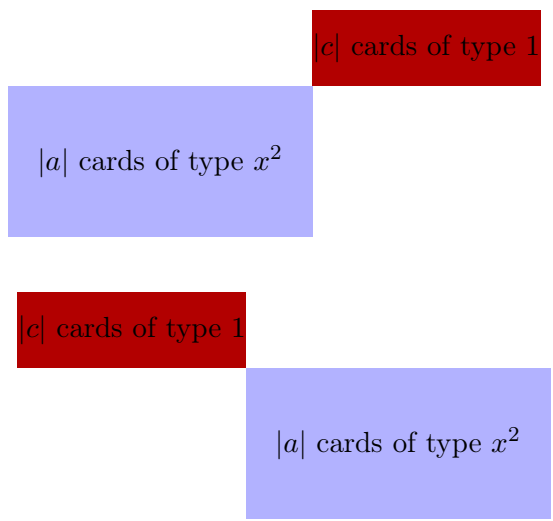
\begin{tikzpicture}
\fill[blue!30!white] (0,0) rectangle (4,2);
\node (x2) at (2,1) {$|a|$ cards of type $x^2$};
\fill[red!70!black] (4,2) rectangle (7,3);
\node (one) at (5.5,2.5) {$|c|$ cards of type $1$};
\end{tikzpicture}
$$

$$
\begin{tikzpicture}
\fill[blue!30!white] (3,0) rectangle (7,2);
\node (x2) at (5,1) {$|a|$ cards of type $x^2$};
\fill[red!70!black] (0,2) rectangle (3,3);
\node (one) at (1.5,2.5) {$|c|$ cards of type $1$};
\end{tikzpicture}
$$

This is equivalent to find divisors of $a$ and $c$.   

Then we set $|b|$ cards of $x$ type, constructing two rectangles with common sides with the previous. If we get a rectangle like one of the following or its rotations
$$
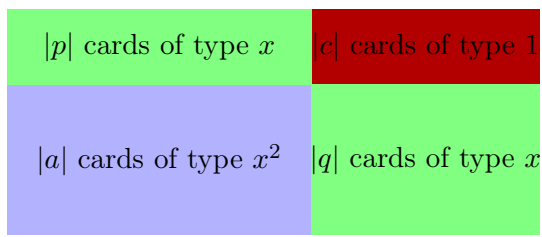
\begin{tikzpicture}
\fill[blue!30!white] (0,0) rectangle (4,2);
\node (x2) at (2,1) {$|a|$ cards of type $x^2$};
\fill[red!70!black] (4,2) rectangle (7,3);
\node (one) at (5.5,2.5) {$|c|$ cards of type $1$};
\fill[green!50!white] (0,2) rectangle (4,3);
\node (p) at (2,2.5) {$|p|$ cards of type $x$};
\fill[green!50!white] (4,0) rectangle (7,2);
\node (q) at (5.5,1) {$|q|$ cards of type $x$};
\end{tikzpicture}
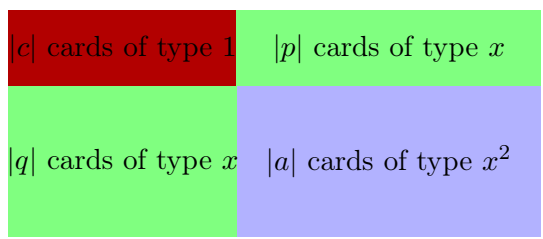
$$

$$
\begin{tikzpicture}
\fill[blue!30!white] (3,0) rectangle (7,2);
\node (x2) at (5,1) {$|a|$ cards of type $x^2$};
\fill[red!70!black] (0,2) rectangle (3,3);
\node (one) at (1.5,2.5) {$|c|$ cards of type $1$};
\fill[green!50!white] (0,0) rectangle (3,2);
\node (q) at (1.5,1) {$|q|$ cards of type $x$};
\fill[green!50!white] (3,2) rectangle (7,3);
\node (p) at (5,2.5) {$|p|$ cards of type $x$};
\end{tikzpicture}
$$

then the polynomial is factorized and the \textit{length} of the sides of the gotten rectangle are its factors.

Otherwise, the rectangles with $x^2$ cards and 1 cards have to be rebuilt, until we get the desired rectangle.

This step is equivalent to find the integers $p$ and $q$ as in Theorem \ref{main}. If we get a rectangle, $p$ and $q$ are given by the rectangles built with cards of $x$ type. If it is not possible to get the desired rectangle, the integers $p$ and $q$ does not exist and the polynomial is not factorable.

By the experience of this factoring procedure with the polynomial box, we were able to establish our main theorem.

Notice that some of the well known methods to factorize a quadratic polynomial with integer coefficients, can be formulated as a consequence of Theorem \ref{main} and its proof.


\begin{corollary}\label{general}
If a polynomial $ax^2 + bx + c \in \mathbb{Z}[x]$ can be factorized over $\mathbb{Z}$ in two linear polynomials, then its factors can be found with the following method:
\begin{enumerate}[i)]
\item Decompose $ac$ in its prime factors, and use this decomposition to find two integers $p$ and $q$ such that $pq=ac$ and $p+q=b$.
\item Rewrite the polynomial and associate: $ax^2 + px + qx + c= (ax^2+px)+(qx+c)$.
\item Use the distributive property to write 
$$ax^2+px=p_1x(q_1x+q_2);$$
where $p_1 = m.c.d. (p,a)$, and $q_1$ and $q_2$ are suitable integers.
\item Find $p_2\in \mathbb{Z}$ such that 
$$qx+c=p_2(q_1x+q_2).$$
\item Use again the distributive property to factorize the polynomial
$$ax^2 + bx + c= p_1x(q_1x+q_2) + p_2(q_1x+q_2)= (p_1x+p_2)(q_1x+q_2).$$
\end{enumerate}
\end{corollary}

\begin{proof}
Theorem \ref{main} guarantees the existence of the integers $p$ and $q$, and its proof, the existence of  $p_1, p_2, q_1, q_2\in \mathbb{Z}$ as required.
\end{proof}

\begin{corollary}\label{sixth}
If a polynomial $x^2 + bx + c \in \mathbb{Z}[x]$ can be factorized over $\mathbb{Z}$ in two linear polynomials, then to factorize it, it is enough to find two integers $p$ and $q$ such that $pq=c$ and $p+q=b$, and
$$x^2 + bx + c = (x+p)(x+q).$$ 
\end{corollary}

\begin{proof}
Follow Corollary \ref{general} with $a=1$.
\end{proof}

\begin{corollary}\label{seventh}
If a polynomial $ax^2 + bx + c \in \mathbb{Z}[x]$, with $a\neq 1$, can be factorized over $\mathbb{Z}$ in two linear polynomials, then its factors can be found with the following method:
\begin{enumerate}[i)]
\item Consider the equivalent expression 
$$\frac{a^2x^2 + abx + ac}{a}.$$
\item Change the variable $y=ax$ and use Corollary \ref{sixth} to factorize the polynomial $a^2x^2 + abx + ac = y^2 + by+ac$.
\item Write the previous factorization in the expression of the first item
$$\frac{a^2x^2 + abx + ac}{a}=\frac{y^2 + by+ac}{a}=\frac{(y + p)(y+q)}{a}=\frac{(ax + p)(ax+q)}{a};$$
then divide the coefficients of the factors by the divisors of $a$ to obtain two factors in $\mathbb{Z}[x]$.
\end{enumerate}
\end{corollary}

\begin{proof}
The polynomial $y^2 + by+ac$ can be factorized because, by Theorem \ref{main}, there exist $p,q \in \mathbb{Z}$ such that $pq=ac$ and $p+q=b$. 

Some divisors of $a$ are divisors of $p$, and the other divisors of $a$ are divisors of $q$, then
$$\frac{(ax + p)(ax+q)}{a} = (p_1x+p_2)(q_1x+q_2);$$
for some $p_1, p_2, q_1, q_2\in \mathbb{Z}$, as we wanted to prove.
\end{proof}

\begin{corollary}\label{squareDif}
A polynomial of the form $a^2x^2-c^2$, with $a, c \in \mathbb{Z}$ can be factorized over $\mathbb{Z}$ in two linear polynomials, moreover
$$a^2x^2-c^2=(ax-c)(ax+c).$$
\end{corollary}

\begin{proof}
Notice that $p=ac$ and $q=-ac$ are integers that satisfy $pq=-a^2c^2$ and $p+q=0$, by Theorem \ref{main} the polynomial $a^2x^2-c^2$ is the product of two linear polynomials in $\mathbb{Z}[x]$.

Following the proof of Theorem \ref{main}, we find $p_1= m.c.d. (p,a^2) = a$, therefore $q_1=a$ and $q_2=c$, then $p_2=-c$, that is
$$a^2x^2-c^2=(ax-c)(ax+c).$$
\end{proof}

\begin{corollary}\label{PerfectSquareTri}
A polynomial of the form $a^2x^2 \pm 2ab + b^2$, with $a, b \in \mathbb{Z}$ can be factorized over $\mathbb{Z}$ in two linear polynomials, moreover
$$a^2x^2 \pm 2ab + b^2=(ax \pm b)^2.$$
\end{corollary}

\begin{proof}
First, take $p=q=ab$, then $pq=a^2b^2$ and $p+q=2ab$, by Theorem \ref{main} the polynomial $a^2x^2+2ab+b^2$ is the product of two linear polynomials in $\mathbb{Z}[x]$. 

Now follow the proof of Theorem \ref{main} to find $p_1= m.c.d. (p,a^2) = a$, therefore $q_1=a$ and $q_2=b$, then $p_2=b$, that is
$$a^2x^2+2ab+b^2 = (ax+b)^2.$$

For $a^2x^2-2ab+b^2$, choose $p=q=-ab$. Notice that the polynomial has two linear factors because $pq=a^2b^2$ and $p+q=-2ab$ (Theorem \ref{main}).

If we follow the proof again, $p_1= m.c.d. (p,a^2) = a$, therefore $q_1=a$ and $q_2=-b$, then $p_2=-b$, that is
$$a^2x^2-2ab+b^2 = (ax-b)^2.$$
\end{proof}

A polynomial with rational coefficients has the same roots as a polynomial in $\mathbb{Z}[x]$, then Theorem \ref{main} allow us to state the next result

\begin{theorem}
A quadratic polynomial with rational coefficients
$$f(x)=\frac{a_1}{a_2}x^2 + \frac{b_1}{b_2}x + \frac{c_1}{c_2};$$
with $a_i, b_i, c_i \in \mathbb{Z}$ for $i\in \{1,2\}$, has rational roots if and only if there exist two integers $p$ and $q$ such that $pq=a_1a_2b_2^2c_1c_2$ and $p+q=b_1a_2c_2$.
\end{theorem}

\begin{proof}
The polynomial $f(x)$ has the same roots as 
$$g(x)=a_1b_2c_2x^2 + a_2b_1c_2x + a_2b_2c_1;$$
which has integer coefficients. 

Recall that a polynomial  in $\mathbb{Z}[x]$ has rational roots if and only if it can be factorized over $\mathbb{Z}$ in two linear polynomials. By Theorem \ref{main} the polynomial $g(x)$ is factorable  if and only if there exist two integers $p$ and $q$ such that $pq=a_1a_2b_2^2c_1c_2$ and $p+q=b_1a_2c_2$, as we wanted to prove. 
\end{proof}

The previous results invite us to take into account the importance of the historical development of algebra and the use of concrete material for its teaching. In addition, they allow observing the impact that these two factors have when clarifying or projecting the path towards the development of a formal and abstract idea related to algebra.

Also, the origin of the research and the results recall the importance of systematizing experiences in the classroom that can uncover important patterns, contributing to the deepening of algebra and its teaching.


Ivon Dorado\\
Departamento de Matem\'aticas, Universidad Nacional de Colombia, Bogot\'a 111211\\
iadoradoc@unal.edu.co

$ $

Ricardo Torres\\
Maestr\'ia en Ense\~nanza de las Ciencias exactas y naturales, Universidad Nacional de Colombia, Bogot\'a 111211\\
riatorresci@unal.edu.co

\end{document}